\documentclass[preprint, 12pt]{elsarticle}

\usepackage[centertags]{amsmath}
\usepackage{multirow,newlfont,amssymb,amsfonts,amsthm,booktabs,longtable,xcolor,ulem}

\newcommand{\changedText}[1]{#1}  

\newcommand{\removedText}[1]{}  

\DeclareMathOperator{\Span}{span}

\newtheorem{thm}{Theorem}[section]
\newtheorem{lem}[thm]{Lemma}

\newdefinition{defn}{Definition}[section]

\theoremstyle{remark}

\allowdisplaybreaks[1] 

\begin{document} 



\begin{frontmatter}


\title{Enhanced group analysis of a semi-linear general bond-pricing equation}


\author[imecc]{Y.~Bozhkov}
\ead{bozhkov@ime.unicamp.br}

\author[cmcc]{S.~Dimas\corref{cor}}
\ead{spawn@math.upatras.gr}

\cortext[cor]{Corresponding author}


\address[imecc]{Instituto de Matem\'atica, Estat\'istica e Computa\c{c}\~ao Cient\'ifica - IMECC\\ 
Universidade Estadual de Campinas - UNICAMP, Rua S\'ergio Buarque de Holanda, $651$\\
$13083$-$859$ - Campinas, SP - Brasil\medskip}

\address[cmcc]{Centro de Matem\'atica, Computa\c{c}\~ao e Cogni\c{c}\~ao - CMCC\\ 
Universidade Federal do ABC - UFABC, Av. dos Estados, $5001$, Bairro Bang\'u\\
$09210$-$580$ - Santo Andr\'e, SP - Brasil}

\begin{abstract}
	The enhanced  group classification of a semi-linear generalization of a general bond-pricing equation is carried out by harnessing the underlying equivalence and additional equivalence transformations. We employ that classification to unearth the particular cases with a larger Lie algebra than the general case and use them to find non trivial invariant solutions under the terminal and the barrier option condition.
  \end{abstract}

\begin{keyword}
Bond-pricing equation \sep Lie point symmetry \sep group classification \sep barrier option \sep computer assisted research 
\MSC[2010] 35A20 \sep 70G65 \sep 70M60 \sep 97M30 \sep 68W30
\end{keyword}

\end{frontmatter}

\section{Introduction}

Recently, Sinkala et al. \removedText{have} introduced a general bond--pricing equation 
\begin{equation}\label{eq:BondPricing}
	u_t+\frac{1}{2}\rho^2x^{2\gamma} u_{xx}+(\alpha+\beta x-\lambda\rho x^\gamma)u_x-x u=0,
\end{equation}
where $\alpha,\beta,\gamma,\rho$ and $\lambda$ are real constants and $u=u(x,t)$ with $x>0$ \removedText{\cite{SiLeaHa2k8}} \changedText{(see \removedText{also} \cite{SiLeaHa2k8, MoKhaMo2k14, MaMaNaMo2k13} and references therein)}. The equation \removedText{have}\changedText{has} an interesting property: it encompasses some of the classical models of Financial Mathematics,  namely the Longstaff model ($\gamma=\delta=1/2,\alpha=\rho^2/4$), the Vacicek model ($\gamma=\delta=0,\beta\ne0$) and the Cox--Ingersoll--Ross model ($\gamma=\delta=1/2,\lambda=0$) \cite{Lo89,Va77,CoIngeRo85}.

In the present work we study --- from a mathematical viewpoint --- a semi-linear generalization of Eq.~\eqref{eq:BondPricing}\footnote{Immediately one can observe that Eq.~\eqref{eq:GBondPricing} contains also the celebrated Black--Scholes--Merton model \cite{BlaScho73} when $\alpha=\lambda=0,\gamma=1$ and $f(x,u) = \beta u$.}, \changedText{that is}
\begin{equation}\label{eq:GBondPricing}
	u_t+\frac{1}{2}\rho^2x^{2\gamma} u_{xx}+(\alpha+\beta x-\lambda\rho x^\delta)u_x-f(x,u)=0,\ \rho\ne0,\ \delta\ne0,1,
\end{equation}
using the analytic machinery provided by symmetry analysis. \changedText{Our aim is to find the complete group classification of Eq.~\eqref{eq:GBondPricing} and use it to propose nonlinear models with potential interest to Financial Mathematics and other scientific areas that use models like \eqref{eq:BondPricing}. In other words, to find specific instances of the function $f(x,u)$ that enlarge the  Lie point symmetry group $\mathcal{G}$ of the most general case: for an arbitrary function $f$ \cite[p. 178]{Olver2k}. In order to achieve that goal we use elements from the enhanced group analysis, namely equivalence and additional transformations, in order to simplify the task at hand \cite{PoEshra2k5,IvaPoSo2k10}.} \removedText{We use elements from the enhanced group analysis, namely equivalence transformations, in order to simplify the main task at hand, the complete group classification of Eq.~\eqref{eq:GBondPricing}. Recall that to perform a complete group classification of a differential equation (or a system of differential equations) involving arbitrary functions or/and parameters, means to find the Lie point symmetry group $\mathcal{G}$ for the most general case, and then to find specific forms of the differential equation for which $\mathcal{G}$ can be enlarged, 
. Having the group classification specific cases of Eq.~\eqref{eq:GBondPricing} can be pinpointed admitting a larger set of symmetries, and consequently, a better likelihood of constructing invariant solutions.  A fact that plays a prominent role when one considers also a set of initial or boundary conditions along the PDE.}

\changedText{Apart from indicating potential models, a larger set of symmetries plays a prominent role when one also considers  a set of initial or boundary conditions along the PDE. Since then, the likelihood of a certain linear combination of symmetries to admit the initial or boundary problem as a whole increases. Hence enabling us to construct invariant solutions for the initial or boundary problem.} 

Since we obtained  Eq.~\eqref{eq:GBondPricing} as a generalization of the  model \eqref{eq:BondPricing} we shall also consider two specific problems associated with it,  
\begin{enumerate}
\item The terminal condition  
	\begin{equation}\label{cond:terminal}
		u(x,T)=1, 
	\end{equation}
	and 
\item The barrier option condition
	\begin{subequations}\label{cond:barrier}
		\begin{align}
			&u(H(t),t)=R(t),\label{cond:barriera}\\ 
			&u(x,T)=\max(x-K,0),	\label{cond:barrierb}
		\end{align}
	\end{subequations}
\end{enumerate}
where $T$ is the terminal time. 

In Financial Mathematics, the former describes the evolution of standard or ``vanilla"  products \cite{BlaScho72,BlaScho73,Me74,Ugur2k8}, while the latter an exotic type of products \cite{ha2k11, Kwo2k8, HaSoLea2k13}. Moreover, a common assumption for the function $H$ --- termed also as the \emph{barrier} function  --- is to have the exponential form
\begin{equation}\label{H0}
H(\tau) = b Ke^{-a \tau},
\end{equation}
where $a\ge0$, $0\le b\le1$ and $\tau=t-T$ \cite[p. 187]{Kwo2k8}. 

As we previously mentioned the key analytical tool used in this work is the symmetry analysis of the Eq.~\eqref{eq:GBondPricing}. One of the advantages of this approach is that it provides a \changedText{well-defined} algorithmic procedure which essentially enables one to find the involved linearizing transformations, conservation laws, invariant solutions, etc. In fact, various works on the classical financial models mentioned above have a connection with the heat equation \cite{GazIbr98, DiAndTsLe2k9}. As probably expected,  Eq.~\eqref{eq:GBondPricing} shares a connection with the heat equation\removedText{, and}\changedText{. M}ore precisely \changedText{we prove here that Eq.~\eqref{eq:GBondPricing} is linked} with the heat equation with nonlinear source,
\begin{equation}\label{eq:HeatNonLinearSource}
	u_t=u_{xx}+f(x,u).
\end{equation}
Furthermore, we show that the exponential form \eqref{H0} usually utilized in the literature is admitted by the symmetries found. A fact that reinforces the \emph{physical} importance of this specific choice by the specialist of the field. 

\changedText{At this point we would like to mention that for}\removedText{For} all the calculations involved the symbolic package SYM for Mathematica was extensively used, both for the interactive manipulation of the found symmetries as well as for determining the equivalence transformation and  the classification of  Eq.~\eqref{eq:GBondPricing} \cite{Dimas2k8}. 

This paper is organized as follows.\removedText{In section 2 the basic concepts of the Lie point symmetry approach to differential equations used in the paper are presented. In section 2 we present the basic concepts of the Lie point symmetry approach and equivalence transformations used in this work. In section 3 the continuous equivalence transformations are found and with their help the complete group classification of Eq.~\eqref{eq:GBondPricing} is obtained.} \changedText{In section 2 we obtain the set of the  continuous equivalence transformations and with their help the complete group classification of Eq.~\eqref{eq:GBondPricing}.} \removedText{In section 3  specific examples of invariant  solutions  under the specific boundary problems studied, the ``vanilla" option and the barrier option, are given.} \changedText{In section 3  we give specific examples of invariant  solutions under the specific boundary problems \changedText{which we study}: the ``vanilla" option and the barrier option.} \removedText{Finally, in section 4 the results of this work are discussed.} Finally, in section 4 we discuss the results of \removedText{this}\changedText{the present} work.

\section{Group classification}

In this section we proceed with the group classification of Eq.~\eqref{eq:GBondPricing} following the same principles as in our previous works, see \cite{BoDi2k14a, BoDi2k14b}.  First, the best representative for the class of equations \eqref{eq:GBondPricing} is obtained utilizing its equivalence algebra. To do that, \changedText{we construct} the continuous part of the equivalence group \removedText{is constructed} and with its help \changedText{we zero out} as many \changedText{arbitrary elements} as possible\removedText{ arbitrary elements}\removedText{are zeroed}.
Initially, observe \changedText{that} the \emph{trivial} transformation 
\begin{equation}\label{mainRes:equivalenceTransformation}
	\begin{aligned}
		&\tilde x = x,\ \tilde t = \frac{\rho^2}{2}t,\ \tilde u = u,\ \tilde\alpha = \frac{2}{\rho^2}\alpha,\ \tilde\beta =\frac{2}{\rho^2}\beta,\  \tilde\gamma=\gamma,\  \tilde\delta=\delta,\ \tilde \lambda = \frac{\sqrt{2}}{\rho}\lambda,\\
		 &\tilde f =\frac{2}{\rho^2}f.
	\end{aligned}
\end{equation}
transforms Eq.~\eqref{eq:GBondPricing}  into the equation 
\begin{equation}\label{main:GBondPricing2}
	u_{\tilde t}+x^{2 \gamma} u_{ xx}+(\tilde  \alpha+\tilde \beta x-\sqrt{2}\tilde \lambda {x}^{\delta}){u}_{x}-\tilde f(x, u)=0.
\end{equation}
We proceed in obtaining the equivalence algebra for Eq.~\eqref{main:GBondPricing2} without making any assumption on the coefficients of the extended infinitesimal operator, 
\begin{multline}\label{igEx}
	\mathfrak{X}={\xi }^1\frac{\partial }{\partial x}+{\xi }^2\frac{\partial }{\partial t}+ \eta_1  \frac{\partial }{\partial u}+\\
	 \eta_2  \frac{\partial }{\partial \alpha}+ \eta_3  \frac{\partial }{\partial \beta}+ \eta_4  \frac{\partial }{\partial \gamma}+ \eta_5  \frac{\partial }{\partial \delta}+ \eta_6  \frac{\partial }{\partial \lambda}+ \eta_7  \frac{\partial }{\partial \rho}+ \eta_8  \frac{\partial }{\partial f},
\end{multline}
where in \eqref{igEx} the coefficients of the operator depend on the extended space $(x,\tilde t,u,\tilde \alpha,\tilde \beta,\gamma,\delta,\tilde \lambda,\tilde f)$\footnote{In other words we look for \emph{generalized equivalence transformations}.}.	
\begin{thm}
The equivalence algebra $\hat{\mathcal{L}}_{\widetilde{\mathcal{E}}}$ of  the class of equations \eqref{main:GBondPricing2} is generated by the following vector fields
{\footnotesize\begin{align}\label{main:firstvec}
	& \mathcal{F}_1({\tilde\alpha} ,{\tilde\beta} ,\gamma ,\delta ,{\tilde\lambda} )\partial _{\tilde t}\\ 
	& \mathcal{F}_2({\tilde\alpha} ,{\tilde\beta} ,\gamma ,\delta ,{\tilde\lambda} )(\tilde f\partial_{\tilde f}+u \partial _u) \\
	\begin{split}
		&\mathcal{F}_3 (x, {\tilde\alpha} ,{\tilde\beta} ,\gamma ,\delta ,{\tilde\lambda} )\partial _u+\\
		&\left(\left({\tilde\alpha} + {\tilde\beta} x -\sqrt{2}  {\tilde\lambda} x^{\delta } \right)\mathcal{F}_{3x} (x, {\tilde\alpha} ,{\tilde\beta} ,\gamma ,\delta ,{\tilde\lambda} ) +x^{2 \gamma }\mathcal{F}_{3xx} (x, {\tilde\alpha} ,{\tilde\beta} ,\gamma ,\delta ,{\tilde\lambda} ) \right)\partial _{\tilde f}
		\end{split}\\
	\begin{split}\label{main:betavector}
		& \mathcal{F}_4({\tilde\alpha} ,{\tilde\beta} ,\gamma ,\delta ,{\tilde\lambda} )\left(\frac{x^{2(1- \gamma) }}{\gamma -1}u\partial _u+4\partial _{{\tilde\beta} }+\right.\\
		&\left. \quad\frac{x^{-2 \gamma } \left(\tilde f x^2+2 (\gamma -1)   \left((2 \gamma -1)x^{2 \gamma } +\sqrt{2}{\tilde\lambda} x^{1+\delta }  -{\tilde\alpha}  x - {\tilde\beta} x^2 \right)u\right)}{\gamma -1}\partial _{\tilde f} \right)
	\end{split}\\
	 \begin{split}\label{main:alphavector}
	 	& \mathcal{F}_5({\tilde\alpha} ,{\tilde\beta} ,\gamma ,\delta ,{\tilde\lambda} )\left(\frac{x^{1-2 \gamma }}{2 \gamma-1}u\partial_u+2\partial _{{\tilde\alpha} }+\right.\\
	 	&\left.\qquad\quad\frac{x^{-1-2 \gamma } \left(\tilde f x^2+ (2 \gamma-1 ) \left( 2\gamma x^{2 \gamma }  +\sqrt{2} {\tilde\lambda}  x^{1+\delta } -{\tilde\alpha} x  -{\tilde\beta} x^2 \right)u\right)}{2 \gamma-1 }\partial _{\tilde f} \right)
		\end{split}\\
		\begin{split}\label{main:lambdavector}
			& \mathcal{F}_6({\tilde\alpha} ,{\tilde\beta} ,\gamma ,\delta ,{\tilde\lambda} )\left(\frac{\sqrt{2}  x^{1-2 \gamma +\delta }}{1-2 \gamma +\delta }u\partial _u+2\partial _{{\tilde\lambda} }+\frac{x^{\delta-1-2 \gamma }}{\delta+1-2 \gamma } \left(\sqrt{2} \tilde f x^2+\right.\right.\\
			&\left.\left. \quad(\delta +1-2 \gamma ) \left(\sqrt{2}{\tilde\alpha}  x +\sqrt{2}{\tilde\beta} x^2  +\sqrt{2} x^{2 \gamma } (\delta-2 \gamma )-2 {\tilde\lambda} x^{1+\delta } \right)u\right)\partial _{\tilde f} \right)
		\end{split}\\
		\begin{split}
			&  \mathcal{F}_7({\tilde\alpha} ,{\tilde\beta} ,\gamma ,\delta ,{\tilde\lambda} )\left(\frac{\sqrt{2}  {\tilde\lambda} x^{1-2 \gamma +\delta } ((1-2 \gamma+\delta ) \log x-1)}{(1-2 \gamma +\delta )^2}u\partial _u+2\partial _{\delta }+\right.\\
			&\qquad{\tilde\lambda} x^{-1+\delta }   \left(\left(u-\frac{\tilde f x^{2(1- \gamma) }}{(1-2 \gamma +\delta )^2}\right)+\frac{x^{-2 \gamma } \log x}{1-2 \gamma +\delta }\left( \tilde f x^2+\right.\right.\\
			& \left.\qquad\qquad (1-2 \gamma +\delta ) \left({\tilde\alpha} x  +{\tilde\beta} x^2  + (\delta -2 \gamma)x^{2 \gamma }-\sqrt{2} {\tilde\lambda} x^{1+\delta }  \right)u\right)\biggr)\partial _{\tilde f} \Biggr)
		\end{split}\\
		\begin{split}
			& \mathcal{F}_8({\tilde\alpha} ,{\tilde\beta} ,\gamma ,\delta ,{\tilde\lambda} )\left(2 x^{\gamma }\partial _x+u x^{-1-\gamma } \left( \gamma x^{2 \gamma }  +\sqrt{2}{\tilde\lambda}  x^{1+\delta } - {\tilde\alpha} x -{\tilde\beta} x^2 \right)\partial _u+\right.\\
			&\quad x^{-3-\gamma } \left(\tilde f \left( \gamma x^{2(1+ \gamma) } +\sqrt{2}{\tilde\lambda}  x^{3+\delta } - {\tilde\alpha} x^3 -{\tilde\beta}  x^4  \right)+u \left( {\tilde\beta} ^2 (\gamma -1)x^4-\right.\right.\\
			&\qquad 2 {\tilde\alpha}  \gamma x^{1+2 \gamma }  + {\tilde\alpha} ^2 \gamma x^2 + {\tilde\alpha}  {\tilde\beta}  (2 \gamma -1)x^3+ \gamma  \left(2-3 \gamma +\gamma ^2\right)x^{4 \gamma }-\\
			&\qquad\qquad \sqrt{2}  (2 \gamma -\delta ) (\delta -1) {\tilde\lambda} x^{1+2 \gamma +\delta }+\sqrt{2}  {\tilde\alpha}  (\delta -2 \gamma) {\tilde\lambda} x^{2+\delta }+\\
			&\qquad\qquad\qquad\left.\left.\left. \sqrt{2} {\tilde\beta}  (1-2 \gamma +\delta ) {\tilde\lambda} x^{3+\delta } -2  (\delta -\gamma) {\tilde\lambda}^2 x^{2(1+ \delta) } \right)\right)\partial _{\tilde f}\right)
		\end{split}\\
		\begin{split}
			& \mathcal{F}_9({\tilde\alpha} ,{\tilde\beta} ,\gamma ,\delta ,{\tilde\lambda} )\left( \frac{2 x}{1-\gamma }\partial _x+4 \tilde t\partial _{\tilde t}+\right.\\
			&\quad\frac{u x^{-2 \gamma } \left({\tilde\alpha} x+{\tilde\beta} x^2  +4 (\gamma -1) x^{2 \gamma } -\sqrt{2} {\tilde\lambda} x^{1+\delta }  \right)}{\gamma -1}\partial _u+\\
			& \ \frac{x^{-1-2 \gamma }}{\gamma -1} \left(\tilde f x^2 \left({\tilde\alpha} +{\tilde\beta} x  -\sqrt{2} {\tilde\lambda} x^{\delta }  \right)-u \left(2  {\tilde\beta} ^2 (\gamma -1)x^3+ {\tilde\alpha}^2 (2 \gamma-1 ) x - \right.\right.\\
			&\qquad 2  {\tilde\alpha}  \gamma  (2 \gamma -1 )x^{2 \gamma }+{\tilde\alpha}  {\tilde\beta}  (4 \gamma -3)x^2-2 {\tilde\beta}  \left(1-3 \gamma +2 \gamma ^2\right)x^{1+2 \gamma } +\\
			&\qquad\qquad\sqrt{2} {\tilde\alpha}  (2-4 \gamma +\delta ) {\tilde\lambda} x^{1+\delta } +\sqrt{2}  {\tilde\beta}  (3-4 \gamma +\delta ) {\tilde\lambda} x^{2+\delta }+\\
			&\qquad\ \left.\left.\left.\sqrt{2}  \left(\delta -2 \gamma \right)\left(\delta -2 \gamma +1 \right) {\tilde\lambda} x^{2 \gamma +\delta } -2  (1-2 \gamma +\delta ){\tilde\lambda} ^2x^{1+2 \delta }\right)\right)\partial _{\tilde f}\right)
		\end{split}\\[12pt]
		\begin{split}\label{main:lastvec}
			&\mathcal{F}_{10}({\tilde\alpha} ,{\tilde\beta} ,\gamma ,\delta ,{\tilde\lambda} )\left(4\partial _{\gamma }-\frac{4 x (1+(\gamma -1) \log x)}{(\gamma -1)^2}\partial _x+\right.\\
			&\quad\frac{x^{-2 \gamma } u}{(2 \gamma -1)^2 (\gamma -1)^2 (1-2 \gamma +\delta )^2}\left(x \left(2 {\tilde\alpha}  \left(2 \gamma ^2-1\right) (1-2 \gamma +\delta )^2+\right.\right.\\
			&\left.(1-2 \gamma)^2 \left({\tilde\beta}  (1-2 \gamma +\delta )^2 x -2 \sqrt{2}  \left(2 \gamma ^2+\delta(2 -4 \gamma +\delta)-1\right) {\tilde\lambda} x^{\delta }\right)\right)+\\
			&\quad2 \left(1-3 \gamma +2 \gamma ^2\right) (1-2 \gamma+\delta ) \left((1-2 \gamma ) (1-2 \gamma +\delta )x^{2 \gamma } + \right.\\
			&\quad\left.\left.{\tilde\alpha}  (1-2 \gamma +\delta )x-\sqrt{2} (2 \gamma -1) (\delta -1) {\tilde\lambda}  x^{1+\delta } \right) \log x\right)\partial _u+\\
			&\quad\frac{x^{-2 (1+\gamma )}}{(2 \gamma -1)^2 (\gamma-1 )^2 (1-2 \gamma +\delta )^2}\times\\
			&\quad\Bigl(2 u (1-2 \gamma )^2 (1-2 \gamma +\delta )^2 \left({\tilde\alpha}  {\tilde\beta}  (1-2 \gamma ) x^3 +(\gamma-1)x^{4 \gamma } -\right.\\
			& \quad{\tilde\beta} ^2 (\gamma -1)x^4+2 {\tilde\beta}  (\gamma -1)^2x^{2(1+ \gamma) } -{\tilde\alpha} ^2 \gamma x^2  +2  {\tilde\alpha}  \left(1-\gamma +\gamma ^2\right)x^{1+2 \gamma }-\\
			&\quad\sqrt{2}  {\tilde\alpha}  (\delta -2 \gamma) {\tilde\lambda} x^{2+\delta } -\sqrt{2} {\tilde\beta}  (1-2 \gamma +\delta ) {\tilde\lambda} x^{3+\delta }  -\\
			&\quad\left.\sqrt{2}  \left(2+2 \gamma ^2-\delta +\delta ^2-2 \gamma  (1+\delta )\right) {\tilde\lambda} x^{1+2 \gamma +\delta } +2  (\delta -\gamma) {\tilde\lambda} ^2x^{2(1+ \delta) }\right)+\\
			&\quad \tilde f x^3 \left(2 {\tilde\alpha}  \left(2 \gamma ^2-1\right) (1-2 \gamma +\delta )^2+(1-2 \gamma )^2 \left({\tilde\beta}  (1-2 \gamma +\delta )^2x-\right.\right.\\
			&\quad\left.\left.2 \sqrt{2} \left(-1+2 \gamma ^2+2 \delta -4 \gamma  \delta +\delta ^2\right) {\tilde\lambda}  x^{\delta }\right)\right)-\\
			&\quad2 x \left(1-3 \gamma +2 \gamma ^2\right) (2 \gamma -\delta-1 ) \left(\tilde f x \left((2 \gamma -1) (2 \gamma -\delta -1)x^{2 \gamma } +\right.\right.\\
			&\quad\left.{\tilde\alpha}  (1-2 \gamma +\delta )x-\sqrt{2}  (2 \gamma -1) (\delta -1) {\tilde\lambda} x^{1+\delta } \right)+\\
			&\quad u (2 \gamma -1) (2 \gamma -\delta -1) \left({\tilde\alpha} ^2 x +{\tilde\alpha}  {\tilde\beta} x^2  -2 {\tilde\alpha}  \gamma x^{2 \gamma } +\sqrt{2}  {\tilde\alpha}  (\delta -2) {\tilde\lambda} x^{1+\delta } +\right.\\
			&\quad\sqrt{2} {\tilde\beta}  (\delta -1) {\tilde\lambda} x^{2+\delta }  -\sqrt{2} (2 \gamma -\delta ) (\delta -1) {\tilde\lambda} x^{2 \gamma +\delta } -\\
			&\quad\quad\left.\left.\left.2  (\delta -1) {\tilde\lambda}^2x^{1+2 \delta }\right)\right) \log x\right)\partial _{\tilde f} \biggr)
		\end{split}
\end{align}}
where $\mathcal{F}_i,\ i=1,\dots,10$ are arbitrary real functions. 
\end{thm}
\begin{proof}
By applying the second order \changedText{augmented} prolongation to the extended system 
\begin{multline*}
	u_{\tilde t}+x^{2\gamma(x,\tilde t,u)} u_{xx}-\tilde f(x,\tilde t,u)+\\
		(\tilde\alpha(x,\tilde t,u)+\tilde \beta(x,\tilde t,u) x-\sqrt{2} \tilde \lambda(x,\tilde t,u) x^{\delta(x,\tilde  t,u)})u_x=0,\\
		{\tilde \alpha}_x={\tilde \alpha}_{\tilde t}={\tilde \alpha}_u={\tilde \beta}_x={\tilde \beta}_{\tilde t}={\tilde \beta}_u=\gamma_x=\gamma_{\tilde t}=\gamma_u=\delta_x=\delta_{\tilde t}=\delta_u=\\
		{\tilde \lambda}_x={\tilde \lambda}_{\tilde t}={\tilde \lambda}_u=\tilde f_{\tilde t}=0,
\end{multline*}
modulo the extended system itself, we get the system of determining equations:
\begin{gather*}
	{\eta _3}_{\tilde f}=0 ,\ {\eta _4}_{\tilde f}=0 ,\ {\eta _5}_{\tilde f}=0 ,\ {\eta _6}_{\tilde f}=0 ,\ {\eta _7}_{\tilde f}=0 ,\ \xi ^2_{\tilde f}=0 ,\ \xi ^2_{{\tilde f\tilde  f}}=0 ,\ {\eta _3}_{u}=0 ,\\ 
	{\eta _4}_{u}=0 ,\ {\eta _5}_{u}=0 ,\ {\eta _6}_{u}=0 ,\ {\eta _7}_{u}=0 ,\ \xi ^2_{u}=0 ,\ \xi ^2_{u\tilde f}=0 ,\ \xi ^2_{uu}=0 ,\ {\eta _1}_{\tilde t}=0 ,\\
	 {\eta _2}_{\tilde t}=0 ,\ {\eta _3}_{\tilde t}=0 ,\ {\eta _4}_{\tilde t}=0 ,\ {\eta _5}_{\tilde t}=0 ,\ {\eta _6}_{\tilde t}=0 ,\ {\eta _7}_{\tilde t}=0 ,\ \xi ^1_{\tilde t}=0 ,\ {\eta _3}_{x}=0 ,\\
 	{\eta _4}_{x}=0 ,\ {\eta _5}_{x}=0 ,\ {\eta _6}_{x}=0 ,\ {\eta _7}_{x}=0 ,\ \xi ^2_{x}=0 ,\ \xi ^2_{\tilde f}=0 ,\ {\eta _1}_{\tilde f}-\tilde f \xi ^2_{\tilde f}=0 ,\\
 	{\eta _1}_{\tilde f}-\tilde f \xi ^2_{\tilde f}=0 ,\ \xi ^1_{\tilde f}-\left({\tilde \alpha} +{\tilde \beta} x -\sqrt{2} {\tilde \lambda} x^{\delta } \right) \xi ^2_{\tilde f}=0 ,\  2 \xi ^2_{\tilde f}-{\eta _1}_{{\tilde f\tilde  f}}+\tilde f \xi ^2_{{\tilde f\tilde  f}}=0 ,\\
	 \xi ^1_{{\tilde f\tilde  f}}-\left({\tilde \alpha} +{\tilde \beta} x -\sqrt{2} {\tilde \lambda} x^{\delta } \right) \xi ^2_{{\tilde f\tilde  f}}=0 ,\ \xi ^1_{u\tilde f}-\left({\tilde \alpha} +{\tilde \beta} x -\sqrt{2} {\tilde \lambda} x^{\delta } \right) \xi ^2_{u\tilde f}=0 ,\\
 	\xi ^1_{\tilde f}-\left({\tilde \alpha} +{\tilde \beta} x -\sqrt{2} {\tilde \lambda} x^{\delta } \right) \xi ^2_{\tilde f}=0 ,\ \xi ^1_{{\tilde f\tilde  f}}-\left({\tilde \alpha} +{\tilde \beta} x -\sqrt{2} {\tilde \lambda} x^{\delta } \right) \xi ^2_{{\tilde f\tilde  f}}=0 ,\\
 	\xi ^1_{uu}-\left({\tilde \alpha} +{\tilde \beta} x -\sqrt{2} {\tilde \lambda} x^{\delta } \right) \xi ^2_{uu}=0 ,\ {\eta _1}_{{\tilde f\tilde  f}}-2 \xi ^2_{\tilde f}-\tilde f \xi ^2_{{\tilde f\tilde  f}}=0 ,\\
 	\left({\tilde \alpha} x+{\tilde \beta} x^2+2 \gamma x^{2 \gamma } -\sqrt{2} {\tilde \lambda} x^{1+\delta } \right) \xi ^2_{\tilde f}+x^{1+2 \gamma } \xi ^2_{x\tilde f}-x \xi ^1_{\tilde f}=0 ,\\
	\left({\tilde \alpha} x+{\tilde \beta} x^2+2 \gamma x^{2 \gamma } -\sqrt{2} {\tilde \lambda} x^{1+\delta } \right) \xi ^2_{u}+x^{1+2 \gamma } \xi ^2_{xu} -x \xi ^1_{u}=0 ,\\
	 \tilde f \xi ^1_{\tilde f}-\tilde f \left({\tilde \alpha} +{\tilde \beta} x -\sqrt{2} {\tilde \lambda} x^{\delta } \right) \xi ^2_{\tilde f}+2 x^{2 \gamma } \left(\xi ^2_{x}-{\eta _1}_{x\tilde f}+\tilde f \xi ^2_{x\tilde f}\right)=0 ,\\
	\left(3 {\tilde \alpha} x+3 {\tilde \beta} x^2+4 \gamma x^{2 \gamma } -3 \sqrt{2} {\tilde \lambda} x^{1+\delta } \right) \xi ^2_{\tilde f}+2 x^{1+2 \gamma } \xi ^2_{x\tilde f} -2 x \xi ^1_{\tilde f}=0 ,\\	
	 \begin{split}
	 	&\left({\tilde \alpha} +{\tilde \beta} x -\sqrt{2} {\tilde \lambda} x^{\delta } \right) {\eta _1}_{\tilde f}-\\
		&\qquad\qquad \tilde f \left({\tilde \alpha} +{\tilde \beta} x -\sqrt{2} {\tilde \lambda} x^{\delta } \right) \xi ^2_{\tilde f}-2 x^{2 \gamma } \left(\xi ^2_{x}-{\eta _1}_{x\tilde f}+\tilde f \xi ^2_{x\tilde f}\right)=0,
	\end{split}\\
	\begin{split}
 		&2 \left({\tilde \beta} x -\sqrt{2}\delta  {\tilde \lambda} x^{\delta }  \right) \xi ^2_{u}+\\
		&\qquad\qquad x \left({\eta _1}_{uu}-\tilde f \xi ^2_{uu}-2 \xi ^1_{xu}+2 {\tilde \alpha}  \xi ^2_{xu}+2 {\tilde \beta} x  \xi ^2_{xu}-2 \sqrt{2} {\tilde \lambda} x^{\delta }  \xi ^2_{xu}\right)=0,
	\end{split}\\
	\begin{split}
 		&\left({\tilde \beta} x -\sqrt{2}  \delta  {\tilde \lambda} x^{\delta } \right) \xi ^2_{\tilde f}+\\
		&\qquad\qquad\qquad x \left({\eta _1}_{u\tilde f}-\tilde f \xi ^2_{u\tilde f}-\xi ^1_{x\tilde f}+{\tilde \alpha}  \xi ^2_{x\tilde f}+{\tilde \beta} x  \xi ^2_{x\tilde f}-\sqrt{2} {\tilde \lambda} x^{\delta }  \xi ^2_{x\tilde f}-\xi ^2_{u}\right)=0 ,
	\end{split}\\
	\begin{split}
	 	&2 x \log x \eta _5+2 \gamma  \xi ^1+\tilde f x \xi ^2_{u}+x \xi ^2_{\tilde t}-2 x \xi ^1_{x}+x^{1+2 \gamma } \xi ^2_{xx}\\
		&\qquad\qquad \qquad\qquad\qquad\qquad  +(3 {\tilde \alpha}  x +3 {\tilde \beta}  x^2 +4 \gamma x^{2 \gamma } - 3 \sqrt{2} {\tilde \lambda} x^{1+\delta })  \xi ^2_{x}=0 ,
	\end{split}\\
	\begin{split}
	 	&\tilde f {\eta _1}_{u}-\eta _2-f^2 \xi ^2_{u}+{\eta _1}_{\tilde t}-\tilde f \xi ^2_{\tilde t}+{\tilde \alpha}  {\eta _1}_{x}+{\tilde \beta} x  {\eta _1}_{x}-\sqrt{2} {\tilde \lambda} x^{\delta }   {\eta _1}_{x}-\tilde f {\tilde \alpha}  \xi ^2_{x}-\tilde f {\tilde \beta} x  \xi ^2_{x}+\\
		&\qquad\qquad \sqrt{2} \tilde f {\tilde \lambda} x^{\delta }   \xi ^2_{x}+x^{2 \gamma } {\eta _1}_{xx} -\tilde f x^{2 \gamma } \xi ^2_{xx}=0 ,
	\end{split}\\
	\begin{split}
		 &\left({\tilde \alpha}^2 x+2 {\tilde \beta} x^{1+2\gamma }  +2 {\tilde \alpha} {\tilde \beta} x^2 +{\tilde \beta} ^2 x^3 -2 \sqrt{2}{\tilde \alpha}  {\tilde \lambda}  x^{1+\delta } -2 \sqrt{2} {\tilde \beta}  {\tilde \lambda}  x^{2+\delta } -\right.\\
		 &\qquad \left.2 \sqrt{2} \delta  {\tilde \lambda} x^{2 \gamma +\delta } +2 {\tilde \lambda} ^2 x^{1+2 \delta }\right) \xi ^2_{\tilde f}-x \left({\tilde \alpha} +{\tilde \beta} x -\sqrt{2} {\tilde \lambda} x^{\delta }  \right) \xi ^1_{\tilde f}+\\
		 &\qquad\qquad 2 x^{1+2 \gamma } \left({\eta _1}_{u\tilde f}-\xi ^2_{u}-\tilde f \xi ^2_{u\tilde f}-\xi ^1_{x\tilde f}+{\tilde \alpha}  \xi ^2_{x\tilde f}+{\tilde \beta} x  \xi ^2_{x\tilde f}-\sqrt{2} {\tilde \lambda} x^{\delta }   \xi ^2_{x\tilde f}\right)=0 ,
	 \end{split}\\
	 \begin{split} 
 		&x \eta _3+x^2 \eta _4-\sqrt{2} {\tilde \lambda}  x^{1+\delta } \log x \eta _6-\sqrt{2} x^{1+\delta } \eta _7+({\tilde \beta} x -\sqrt{2}\delta  {\tilde \lambda} x^{\delta } )  \xi ^1-\tilde f x \xi ^1_{u}+\\
		&\quad ({\tilde \alpha} \tilde f x +{\tilde \beta} \tilde f x^2 -\sqrt{2} {\tilde \lambda} \tilde f x^{1+\delta })  \xi ^2_{u}-x \xi ^1_{\tilde t}+({\tilde \alpha}  x +{\tilde \beta}  x^2 -\sqrt{2} {\tilde \lambda} x^{1+\delta })  \xi ^2_{\tilde t}-\\
		&\quad ({\tilde \alpha}  x -{\tilde \beta}  x^2  +\sqrt{2} {\tilde \lambda} x^{1+\delta } ) \xi ^1_{x}+({\tilde \alpha}^2 x+2 {\tilde \beta} x^{1+2\gamma }   +2 {\tilde \alpha} {\tilde \beta} x^2 +{\tilde \beta} ^2 x^3 ) \xi ^2_{x}-\\
		&\ (2 \sqrt{2} {\tilde \alpha}  {\tilde \lambda} x^{1+\delta }  +2 \sqrt{2} {\tilde \beta}  {\tilde \lambda} x^{2+\delta } +2 \sqrt{2}\delta  {\tilde \lambda} x^{2 \gamma +\delta } - 2{\tilde \lambda} ^2 x^{1+2 \delta }  )\xi ^2_{x}+2 x^{1+2 \gamma } {\eta _1}_{xu}-\\
		&\quad 2 \tilde f x^{1+2 \gamma } \xi ^2_{xu}-x^{1+2 \gamma } \xi ^1_{xx}+( {\tilde \alpha} x^{1+2 \gamma } +{\tilde \beta} x^{2(1+ \gamma) }  -\sqrt{2}{\tilde \lambda} x^{1+2 \gamma +\delta }  ) \xi ^2_{xx}=0 . 
	\end{split}
\end{gather*}
Solving this system we get the equivalence algebra $\hat{\mathcal{L}}_{\widetilde{\mathcal{E}}}$.
\end{proof}
From the previous system of overdetermined equations three special cases emerge: $\gamma=1,\, \gamma=1/2$ and $\delta=2\gamma-1$. We proceed with the most generic case, $(\gamma-1)(\gamma-1/2)(\delta-2\gamma+1)\ne0$, and afterwards we  treat each special case separately.
\begin{lem}
	For $(\gamma-1)(\gamma-1/2)(\delta-2\gamma+1)\ne0$ there is an equivalence transformation that zeroes the parameters $\tilde\alpha,\tilde\beta,\tilde\lambda$:
	\begin{align*}
		\hat u &=\mathcal{A}_1 u,\\
		\hat f&= \frac{1}{4} x^{-2 \gamma -1} \mathcal{A}_1\left(4x^{2 \gamma +1} \tilde f+  \left(2 \tilde\beta  x^{2 \gamma +1}-4 \tilde\alpha  \gamma  x^{2 \gamma }-4 \tilde\beta  \gamma  x^{2 \gamma +1}+2 \tilde\lambda ^2 x^{2 \delta +1}+\right.\right.\\
			&\qquad\qquad\left.\left.\tilde\beta ^2 x^3+2 \tilde\alpha  \tilde\beta  x^2-2 \sqrt{2} \tilde\lambda  x^{\delta } \left((\delta -2 \gamma ) x^{2 \gamma }+x (\tilde\alpha +\tilde\beta  x)\right)+\tilde\alpha ^2 x\right)u\right),
	\end{align*}
where
$$
	\mathcal{A}_1 = \exp \left(\frac{1}{4} x^{1-2 \gamma } \left(\frac{2 \tilde\alpha }{1-2 \gamma }-\frac{2 \sqrt{2} \tilde\lambda}{\delta -2 \gamma  +1}x^{\delta }-\frac{\tilde\beta }{\gamma -1}x\right)\right).
$$	
\end{lem}
\begin{proof}
	By exponentiating the vectors  \eqref{main:betavector}, \eqref{main:alphavector} and \eqref{main:lambdavector} (after setting the arbitrary functions to unity) we have that $\hat\beta =\tilde \beta+4\zeta_0,\ \hat\alpha=\tilde\alpha+2\zeta_1,\ \hat\lambda=\tilde\lambda+2\zeta_2$, respectively. By setting $\zeta_0=-\tilde\beta/4,\ \zeta_1=-\tilde\alpha/2$ and $\zeta_2=-\tilde\lambda/2$ and substituting to the rest of the transformations we get the aforementioned result.
\end{proof}
Applying the above transformation to Eq.~\eqref{main:GBondPricing2} we get
\begin{equation}\label{main:GBondPricing3}
	\hat u_{\tilde t}+x^{2\gamma} \hat u_{xx}-\hat f(x,\hat u)=0.
\end{equation}
Once again we repeat the same process, now for Eq.~\eqref{main:GBondPricing3}, also assuming that $\xi^i=\xi^i(x,\tilde t,u), \eta_1=\eta_1(x,\tilde t,u)$. For brevity, we skip the intermediate steps and give directly the equivalence transformations in the following lemma.
\begin{lem}
	The continuous part of the equivalence group, $\hat{\mathcal{E}}_\mathcal{C}$, of Eq.~\eqref{main:GBondPricing3} consists of the transformations
	\begin{align*}
		&\bar x =  x^\frac{1}{\zeta_2^2},\ \bar t = \zeta_0+ \frac{1}{\zeta_2^4}\tilde t,\ \bar u = x^{\frac{1}{2}\left(\frac{1}{\zeta_2^2}-1\right)}(\zeta_1 \hat u +  \mathcal{F}(x)),\ \bar\gamma=1+\zeta_2^2(\gamma-1),\\
		 &\bar f = \frac{1}{4} x^{\frac{1}{2}\left(\frac{1}{\zeta_2^2}-5\right)}\left( (\zeta_2^4-1)x^{2\gamma}(\zeta_1 \hat u +  \mathcal{F}(x))+4\zeta_2^4x^2(\zeta_1 \hat f+
		 x^{2\gamma}\mathcal{F}^{\prime\prime}(x))\right),
	\end{align*}
	where $\zeta_0,\zeta_1,\zeta_2$ are arbitrary constants with $\zeta_1\zeta_2\ne0$ and $\mathcal{F}$ is an arbitrary real function.
\end{lem}
Immediately,  one can choose a specific equivalence transformation\footnote{At this point we have also included the discrete equivalence transformation $t\rightarrow -t, f\rightarrow -f$.} that zeroes $\bar\gamma$:
\begin{equation}\label{mainRes:equivalenceTransformation2}
	\begin{aligned}
		&\bar x =  x^{1-\gamma},\  \tau = -(\gamma-1)^2 \tilde t,\ \phi = x^{-\frac{\gamma}{2}}\hat u,\\ &\bar f =-\frac{1}{4(\gamma-1)^2} x^{-\frac{1}{2}(\gamma+4)}\left( \gamma(2-\gamma)x^{2\gamma}\hat u+4x^2 \hat f\right).
	\end{aligned}
\end{equation}
Via this transformation Eq.~\eqref{main:GBondPricing3} transforms into the equation
\begin{equation}\label{main:GBondPricing4}
	\phi_{\tau}-\phi_{\bar x\bar x}-\bar f(\bar x,\phi)=0,
\end{equation}
the heat equation with nonlinear source. Hence \changedText{we reduced} the problem of the complete group classification of Eq.~\eqref{eq:GBondPricing} --- for $(\gamma-1)(\gamma-1/2)(\delta-2\gamma+1)\ne0$ ---  \removedText{is reduced} to the group classification of the heat equation with nonlinear source, \eqref{eq:HeatNonLinearSource}. 

Before continuing with the group classification of Eq.~\eqref{eq:HeatNonLinearSource} we give, omitting the detailed calculations, the equivalence transformations for each one of the three special cases.

\subsection{The case $\gamma=1$}

\begin{lem}
	There is a equivalence transformation that zeroes the parameters $\tilde\alpha,\tilde\beta,\tilde\lambda$:
	\begin{align*}
		\hat u &= \mathcal{A}_2u,\\
		\hat f&= \frac{1}{4x^2}\mathcal{A}_2 \left(4  x^2 \tilde f+ \left(\tilde\alpha ^2+2 \tilde\lambda ^2 x^{2 \delta }-2 \sqrt{2} \tilde\lambda x^{\delta} (\tilde\alpha +(\tilde\beta +\delta -2)x)+\tilde\beta ^2 x^2-\right.\right.\\
		&\left.\left.\qquad\qquad\ 2 \tilde\beta  x^2+2 \tilde\alpha  \tilde\beta  x-4 \tilde\alpha  x\right)u\right),
	\end{align*}
	where
	$$
		\mathcal{A}_2= x^{\tilde\beta /2} \exp\left(-\frac{\tilde\alpha +\frac{\sqrt{2} \tilde\lambda }{\delta -1} x^{\delta}}{2 x}\right).
	$$	
\end{lem}
Applying this transformation we arrive again to Eq.~\eqref{main:GBondPricing3}, now for $\gamma=1$. For this case an additional equivalence transformation\footnotemark[\value{footnote}] exists,
\begin{equation*}
	\bar x =  \log x,\ \tau=-\tilde t,\ \phi = x^{-1/2}\hat u,\ \bar f = -\frac{1}{\sqrt{x}}(\frac{1}{4}\hat u+\hat f) , 
\end{equation*}
 that turns it into Eq.~\eqref{main:GBondPricing4}. 
 
\subsection{The case $\gamma=1/2$}

\begin{lem}
	There is a equivalence transformation that zeroes the parameters $\tilde\alpha,\tilde\beta,\tilde\lambda$:
	\begin{align*}
		\hat u &= \mathcal{A}_3u,\\
		\hat f&= \frac{1}{4x}\mathcal{A}_3 \left(4 x\tilde f +(\tilde\alpha -2) \tilde\alpha  u+u \left(2 \tilde\lambda ^2 x^{2 \delta }-2 \sqrt{2} \tilde\lambda   x^{\delta } (\tilde\alpha +\delta -1+\tilde\beta  x)+\right.\right.\\
		&\qquad\qquad\left.\left.\tilde\beta  x (2 \tilde\alpha +\tilde\beta  x)\right)\right),
	\end{align*}
	where
	$$
		\mathcal{A}_3= x^{\tilde\alpha /2} \exp\left(\frac{\tilde\beta }{2}x-\frac{\tilde\lambda  }{\sqrt{2} \delta } x^{\delta }\right).
	$$	
\end{lem}
Similarly, utilizing this transformation we reach Eq.~\eqref{main:GBondPricing3}, now for $\gamma=1/2$. By using now the equivalence transformation \eqref{mainRes:equivalenceTransformation2} for $\gamma=1/2$, 
\begin{equation*}
		\bar x =  \sqrt{x},\ \tau =  -\tilde t/4,\ \phi =  x^{-1/4}\hat u,\ \bar f = - x^{-\frac{5}{4}}\left( \frac{3}{4}\hat u+4x \hat f\right),
\end{equation*}
we again arrive at Eq.~\eqref{main:GBondPricing4}. 

\subsection{The case $\delta=2\gamma-1\ne0,1$}

\begin{lem}
	There is a equivalence transformation that zeroes the parameters $\tilde \alpha,\tilde \beta,\tilde\lambda$:
	\begin{align*}
		\hat u &= \mathcal{A}_4u,\\
		\hat f&= \frac{1}{4} x^{-2 \gamma -2}\mathcal{A}_4\left(4 x^{2 \gamma +2}\tilde f+\left(2 \tilde\lambda  \left(\tilde\lambda +\sqrt{2}\right) x^{4 \gamma }+x^2 (\tilde\alpha +\tilde\beta  x)^2-\right.\right.\\
			&\left.\left.\qquad2 x^{2 \gamma +1} \left(\tilde\alpha  \left(2 \gamma +\sqrt{2} \tilde\lambda \right)+\tilde\beta  \left(2 \gamma +\sqrt{2} \tilde\lambda -1\right)x\right)\right)u\right),
	\end{align*}
	where
	$$
		\mathcal{A}_4= x^{-\frac{\tilde\lambda }{\sqrt{2}}} \exp \left(-\frac{1}{4} x^{1-2 \gamma } \left(\frac{2 \tilde\alpha }{2 \gamma -1}+\frac{\tilde\beta  x}{\gamma -1}\right)\right).
	$$	
\end{lem}
Likewise, using this transformation we arrive to Eq.~\eqref{main:GBondPricing3} and by using again the equivalence transformation \eqref{mainRes:equivalenceTransformation2} we reach once more the Eq.~\eqref{main:GBondPricing4}. 

\subsection{Group classification of the heat equation with nonlinear source}

As  we shown in the previous sections Eq.~\eqref{eq:GBondPricing} is linked via a series of point transformations to the heat equation with nonlinear source \eqref{main:GBondPricing4}. Hence, the problem of the group classification of \eqref{eq:GBondPricing} is reduced to obtaining the group classification of  \eqref{main:GBondPricing4}. This classification was done in \cite{BoDi2k14b} and is  included, as a subset, in the work of Zhdanov \emph{et al.} \cite{ZhdaLa99}. In their work, the group classification of the heat conductivity equation with a nonlinear source
\begin{equation*}
	u_t=u_{xx}+F(t,x,u,u_x)
\end{equation*}
is performed taking advantage  of the fact that the abstract Lie algebras of dimensions up to  five are already classified. 

\section{Examples of invariant solutions}

Having accessible the complete group classification for Eq.~\eqref{main:GBondPricing4} ---  and consequently for Eq.~\eqref{eq:GBondPricing} --- we can look for invariant solutions under the terminal condition \eqref{cond:terminal} and the barrier option condition \eqref{cond:barrier}: Given a specific algebra from the classification in \cite{BoDi2k14b} the appropriate subalgebra (and  the functions $H(t), R(t)$ for the barrier option problem) admitted by each problem are determined utilizing the conditions \eqref{cond:terminal}, \eqref{cond:barrier} adapted now for Eq.~\eqref{main:GBondPricing4}. 

Namely,
\begin{equation}\label{main:terminala}
	\left.\mathfrak{X}(\tau-T^\prime)\right\rvert_{\tau=T^\prime}\equiv0,
\end{equation}
\begin{equation}\label{main:terminalb}
	\left.\mathfrak{X}\left(\phi-\Phi(X^{-1}(\bar x),1)\right)\right\rvert_{\tau=T^\prime,\, \phi=\Phi(X^{-1}(\bar x),1)}\equiv0,
\end{equation}
where $T^\prime = \Psi(T)$, and
\begin{equation}\label{main:barriera}
	\left.\mathfrak{X}\left(\bar x-X(H^*(\tau))\right)\right\rvert_{\bar x=X(H^*(\tau))}\equiv0,
\end{equation}
\begin{equation}\label{main:barrierb}
	\left.\mathfrak{X}\left(\phi-\Phi(H^*(\tau) ,R^*(\tau) )\right)\right\rvert_{\phi=\Phi(H^*(\tau) ,R^*(\tau) )}\equiv0,
\end{equation}
where $\mathfrak{X}$ denotes the Lie algebra chosen and  $H^*(\tau) = H(\Psi^{-1}(\tau))$, $R^*(\tau) = R(\Psi^{-1}(\tau))$ with $\bar x =X(x),\ \tau = \Psi(t),$ $\phi = \Phi(x,u)$ denote the point transformation in each case.
Finally, by using the found subalgebra \changedText{we construct} similarity solutions \removedText{are constructed} as per usual. 

In \cite{BoDi2k14b} we \removedText{have} \changedText{illustrated} in detail the process, here we avoid exposing the cumbersome calculations and  present a few illustrative examples that by no means exhaust all the possible solutions that can be found using the classification of Eq.~\eqref{main:GBondPricing4}.

\subsection{The terminal condition}

Due to the restrictions imposed to the admitted symmetries by this condition the chances of obtaining an nontrivial analytic solution, a solution depending explicitly on both variables $\bar x,\tau$,  using less than a \removedText{four}\changedText{three}-dimensional algebra are slim. Hence and the fact that the following examples are restricted to the \changedText{higher-dimensional} algebras only.

\subsubsection{$(\gamma-1)(\gamma-1/2)(\delta-2\gamma+1)\ne0$}

By using the Lie algebra $A^4_4$,
$$
	 \Span\left(\partial_{\tau},  e^{-\frac{2\beta}{\rho^2}\tau } \phi\partial _\phi, \frac{2\beta}{\rho^2} \partial _{\bar x}+B \phi\partial _\phi,\ 2 e^{-\frac{2\beta}{\rho^2} \tau} \partial _{\bar x}+e^{-\frac{2\beta}{\rho^2} \tau} (\frac{2\beta}{\rho^2} \bar x+2 B \tau)\phi \partial _\phi\right),
$$
with $A=-2\beta/\rho^2$, we get the similarity solution
\begin{multline*}
	u(x,t)=\exp\left(\frac{1}{8 \beta ^3 \rho ^2}\left(\rho ^2 \left(\beta ^2 B^2 \rho ^6 \omega ^2-2 B^2 \rho ^6 \left(e^{\beta  \omega }-1\right)+2 \beta  B^2 \rho ^6 \omega -\right.\right.\right.\\
	\left.\left.\left.4 \beta ^3 \left(e^{\beta  \omega }+B \rho ^2 x \omega -1\right)\right)-4 \alpha ^2 \beta ^2 \left(e^{\beta  \omega }-1\right)-4 \alpha  \beta  B \rho ^4 \left(\beta  \omega -e^{\beta  \omega }+1\right)\right)\right),
\end{multline*}
where $\omega=t-T$  for $f(x,u) =-\frac{1}{2 \rho ^2} u \left(\alpha ^2+\beta  \rho ^2+B \rho ^4 x-2 \beta  \rho ^2 \log\lvert u\rvert\right)$ and $\gamma=\lambda=0$.

\subsubsection{$\gamma=1$}

By using the Lie algebra $A^3_{3,8}$,
$$
 	\Span\left(\partial_{\tau},  2 \bar x\partial _{\bar x}+4 \tau\partial_{\tau}, 4 \bar x\tau\partial _{\bar x}+4\tau^2\partial_{\tau}-\bar x^2 \phi\partial _\phi\right),
$$
with $A=1/2,\ \Gamma=0$, we get the similarity solution
$$
	u(x,t)=e^{\frac{\log ^2x}{2 \rho ^2 (t-T)}},
$$
for $f(x,u) =\frac{\rho ^2 u \log \lvert u\rvert}{\log ^2x}$ and $\alpha = \lambda\rho,\beta=\frac{\rho^2}{2},\delta=0$.

\subsubsection{$\gamma=1/2$}

Again, using the Lie algebra $A^3_{3,8}$ ,
\begin{multline*}
 	\Span\left(\partial_{\tau},  2 \bar x\partial _{\bar x}+4 \tau\partial_{\tau}+\frac{(4\alpha-\rho^2)\phi}{\rho^2}\partial_\phi,\right. \\
		\left.4 \bar x\tau\partial _{\bar x}+4\tau^2\partial_{\tau}-\left(\bar x^2 + \frac{2(\rho^2-4\alpha)t}{\rho^2}\right)\phi\partial_\phi\right),
\end{multline*}
with $A=2\alpha/\rho^2,\ B=\Gamma=0$, we get the similarity solution
$$
	u(x,t)=e^{-\frac{2 x (t-T)}{\rho ^2  T t}},
$$
for $f(x,u) =\frac{\rho ^2 u \log\lvert u\rvert}{4 x}-\frac{2 x u }{\rho ^2 T^2}$ and $\alpha=\frac{\rho^2}{4},\beta=\frac{2}{T},\lambda=0$.

\subsubsection{$\delta=2\gamma -1$}

By using the Lie algebra $A^4_4$ ,
$$
	 \Span\left(\partial_{\tau},  e^{A\tau } \phi\partial _\phi, A \partial _{\bar x}-B \phi\partial _\phi,\ 2 e^{A \tau} \partial _{\bar x}+e^{A \tau} (2 B \tau - A\bar x)\phi \partial _\phi\right),
$$
we get the similarity solution
\begin{multline*}
	u(x,t)=\exp \left(\frac{B}{A^3} \left(B e^{A (\gamma -1)^2 \rho ^2 (T-t)}-A^2 x^{1-\gamma }+\right.\right.\\
	 \left.\left.A x^{-\gamma } e^{\frac{1}{2} A (\gamma -1)^2 \rho ^2 (T-t)} \left(A x+B (\gamma -1)^2 \rho ^2 (t-T) x^{\gamma }\right)-B\right)\right),
\end{multline*}
for $f(x,u) =-\frac{1}{2} (\gamma -1)^2 \rho ^2 u \left(A \log\lvert u\rvert +B x^{1-\gamma}\right)$ and $\alpha,\beta=0,\ \lambda=-\frac{\gamma\rho}{2}$.

\subsection{The condition for the barrier option}

Contrary to the terminal condition, for the barrier option --- due to the two arbitrary functions $H(t), R(t)$ --- any choice of  Lie Algebras from the classification of \eqref{main:GBondPricing4} can be utilized in order to obtain nontrivial similarity solutions. Furthermore, the majority of the symmetries admit the form of the barrier function $H(t)$, \eqref{H0}, \changedText{found}\removedText{used} in the literature. \changedText{Another indication that indeed symmetries have the ability to highlight the significance of specific choices}\removedText{A strong indication that this choice for the barrier function has a significant physical --- in our case financial --- meaning}.

\subsubsection{$(\gamma-1)(\gamma-1/2)(\delta-2\gamma+1)\ne0$}

By using the Lie algebra $A^1_{2,2}$,
$$
	 \Span\left(\partial_{\tau}, e^{\tau } x\partial _x+2e^{\tau}\partial _{\tau} -  \frac{e^{\tau}}{4}  \left(x^2 -2 \right)\phi\partial _\phi\right),
$$
with $A=-1/2$, we get the similarity solution
$$
	u(x,t)=\frac{\sqrt{x}}{\mathcal A} \log\lvert\frac{x^{8 (\gamma -1)} e^{-2 (\gamma -1)^2 \rho ^2 t}}{256 \mathbf{c}}+\mathbf{c}x^{8(1- \gamma) } e^{2 (\gamma -1)^2 \rho ^2 t}\rvert,
$$
for 
\begin{multline*}
 	f(x,u) =\frac{x^{-2 \gamma -\frac{5}{2}}}{32 \rho ^2} \times 	\left(\frac{16 (\gamma -1)^2 \rho ^4 x^{4 \gamma +1}}{\mathcal A} \exp \left(-\frac{2 \mathcal A u}{\sqrt{x}}\right)+\right.\\
	u \left(32 \alpha  \gamma  \rho ^2 x^{2 \gamma +\frac{3}{2}}+32 \alpha  \lambda  \rho  x^{\delta +\frac{5}{2}}-16 \alpha ^2 x^{5/2}-32 \alpha  \beta  x^{7/2}++32 \beta  \lambda  \rho  x^{\delta +\frac{7}{2}}+\right.\\
	 x^{9/2} \left((\gamma -1)^2 \rho ^4-16 \beta ^2\right)-8 \rho ^2 x^{2 \gamma +\frac{5}{2}} \left(\beta  (2-4 \gamma )+(\gamma -1)^2 \rho ^2\right)-\\
	\left.\left.16 \lambda  \rho ^3 (2 \gamma -\delta ) x^{2 \gamma +\delta +\frac{3}{2}}-4 \rho ^4 x^{4 \gamma +\frac{1}{2}}-16 \lambda ^2 \rho ^2 x^{2 \delta +\frac{5}{2}}\right)\right),
\end{multline*}
where 
$$
	\mathcal A=\exp\left(\frac{1}{8} x^{1-2 \gamma } \left(\frac{4 \left(\frac{2 \alpha }{1-2 \gamma }-\frac{2 \lambda  \rho  x^{\delta }}{-2 \gamma +\delta +1}-\frac{\beta  x}{\gamma -1}\right)}{\rho ^2}+x\right)\right), 
$$ 
$\mathbf{c}\ne0$ is a constant and $ H(t) = b K e^{\frac{1}{4} (\gamma -1) \rho ^2 (t-T)}$.

\subsubsection{$\gamma=1$}

By using the Lie algebra $A^9_{3,5}$,
\begin{multline*}
 	\Span\Biggl(\partial_{\tau}, \partial_x+\frac{a}{\rho^2}\phi\partial_\phi,\\
	 \left(x-\frac{2a}{\rho^2} \tau\right)\partial_x+2\tau\partial_{\tau}-\left(\left(2-\frac{a}{\rho^2} x+\frac{2 a^2}{\rho^4}\tau\right) \phi-\frac{a^2}{\rho^4}e^{-B x} \right)\partial_\phi  \Biggr),
\end{multline*}
with $B=-a/\rho^2$, we get the similarity solution
$$
	u(x,t)=\frac{x^{\frac{a-\beta }{\rho ^2}+\frac{1}{2}} \left(a^4 t^2+a^2  (\log x-2 a t)\log x+12 \rho ^4\right) e^{\frac{\alpha +\frac{\lambda  \rho  x^{\delta }}{\delta -1}}{\rho ^2 x}}}{2 \rho ^4 (\log x-a t)^2},
$$
for 
\begin{multline*}
	f(x,u) =\frac{1}{8} \left(\frac{a^4 x^{\frac{a}{\rho ^2}-\frac{\beta }{\rho ^2}+\frac{1}{2}} e^{\frac{\alpha  (\delta -1)+\lambda  \rho  x^{\delta }}{(\delta -1) \rho ^2 x}}}{\rho ^6}+4 \rho ^2 u^2 x^{-\frac{a}{\rho ^2}+\frac{\beta }{\rho ^2}-\frac{1}{2}} e^{\frac{-\alpha  \delta +\alpha -\lambda  \rho  x^{\delta }}{(\delta -1) \rho ^2 x}}-\right.\\
	\frac{u }{\rho ^2 x^2}\left(4 \alpha ^2-4 \lambda  \rho  x^{\delta +1} \left(2 \beta +(\delta -2) \rho ^2\right)+4 \lambda ^2 \rho ^2 x^{2 \delta }-\right.\\
	\left.\left.8 \alpha  \left(\lambda  \rho  x^{\delta }+x \left(\rho ^2-\beta \right)\right)+x^2 \left(\rho ^2-2 \beta \right)^2\right)\right)
\end{multline*}
and $H(t) = b K e^{a (t-T)}, a<0$.

\subsubsection{$\gamma=1/2$}

By using the Lie algebra $A^2_{3,8}$,
\begin{multline*}
 	\Span\left(\partial_{\tau},  \frac{4 a}{\rho ^2} e^{- \frac{8 a}{\rho ^2} \tau} x\partial _x- e^{-\frac{8 a}{\rho ^2} \tau }\partial_{\tau}+\right.\\
			\frac{2 a}{\rho ^2} e^{-\frac{8 a}{\rho ^2}\tau}  \left(2\left(\frac{2 a}{\rho ^2}+ \Gamma\right) \Delta e^{-\frac{1}{2} \Gamma x^2 } x^{\frac{5}{2}} +\left(\frac{4 a}{\rho ^2} x^2+1\right)\phi\right)\partial _\phi,\frac{4 a}{\rho ^2} e^{\frac{8 a}{\rho ^2} \tau} x\partial _x+\\
			\left. e^{\frac{8 a}{\rho ^2}\tau }\partial_{\tau}-
			\frac{2 a}{\rho ^2} e^{\frac{8 a}{\rho ^2}\tau} \left(2\left(\frac{2 a}{\rho ^2}-\Gamma\right) \Delta e^{-\frac{1}{2} \Gamma x^2} x^{\frac{5}{2}} + \left(\frac{4 a}{\rho ^2} x^2-1\right)\phi\right)\partial _\phi\right),
\end{multline*}
with $A=1,\ B=-\frac{16 a^2}{\rho ^4}$, we get the similarity solution
\begin{multline*}
	u(x,t)=x^{\frac{1}{4}-\frac{\alpha }{\rho ^2}} e^{\frac{\lambda  \rho  x^{\delta }-\beta  \delta  x}{\delta  \rho ^2}}\times\\
	 \left(x^{\frac{1}{8} \left(4 a t-\frac{\mathbf{c} \rho ^2}{a}+2\right)} \exp \biggl(\frac{1}{64} \biggl(\frac{\mathbf{c}^2 \rho ^4}{a^2}+16 a^2 t^2-\frac{64 a x}{\rho ^2}-16 \Gamma -8 \mathbf{c} \rho ^2 t+\right.\\
	 \left.16 \log ^2x+28\biggr)\biggr)-\sqrt[4]{x} \Delta e^{-\frac{1}{2} \Gamma x}\right),
\end{multline*}
where $\mathbf{c}$ is a constant, for 
\begin{multline*}
	f(x,u) =\frac{1}{32 \rho ^2}\left(16 a^2 \left(u x+\Delta x^{\frac{3}{2}-\frac{\alpha }{\rho ^2}} e^{\frac{\lambda  x^{\delta }}{\delta  \rho }-\frac{\beta  x}{\rho ^2}-\frac{\Gamma x}{2}}\right)-32 \alpha  \beta  u-8 \Gamma \rho ^4 u+\right.\\
	32 \alpha  \lambda  \rho  u x^{\delta -1}-16 \lambda ^2 \rho ^2 u x^{2 \delta -1}-16 \lambda  \rho ^3 u x^{\delta -1}+16 \delta  \lambda  \rho ^3 u x^{\delta -1}+32 \beta  \lambda  \rho  u x^{\delta }+\\
	16 \rho ^4 \left(\frac{u}{x}+\Delta x^{-\frac{\alpha }{\rho ^2}-\frac{1}{2}} e^{\frac{\lambda  x^{\delta }}{\delta  \rho }-\frac{\beta  x}{\rho ^2}-\frac{\Gamma x}{2}}\right) \log \lvert  x^{\frac{\alpha }{\rho ^2}-\frac{1}{2}} e^{-\frac{\lambda  x^{\delta }}{\delta  \rho }+\frac{\beta  x}{\rho ^2}+\frac{\Gamma x}{2}}u+\Delta\rvert-\frac{16 \alpha ^2 u}{x}+\\
	\frac{16 \alpha  \rho ^2 u}{x}-16 \beta ^2 u x+\frac{4 \Gamma  \rho ^4 u}{x}-\frac{3 \rho ^4 u}{x}+4 \Gamma  \rho ^4 \Delta x^{-\frac{\alpha }{\rho ^2}-\frac{1}{2}} e^{\frac{\lambda  x^{\delta }}{\delta  \rho }-\frac{\beta  x}{\rho ^2}-\frac{\Gamma x}{2}}-\\
	\left.4 \Gamma^2 \rho ^4 \Delta x^{\frac{3}{2}-\frac{\alpha }{\rho ^2}} e^{\frac{\lambda  x^{\delta }}{\delta  \rho }-\frac{\beta  x}{\rho ^2}-\frac{\Gamma x}{2}}+\rho ^4 \Delta x^{-\frac{\alpha }{\rho ^2}-\frac{1}{2}} e^{\frac{\lambda  x^{\delta }}{\delta  \rho }-\frac{\beta  x}{\rho ^2}-\frac{\Gamma x}{2}}\right)
\end{multline*}
and $H(t) = b K e^{-a (t-T)}, a>0$.

\subsubsection{$\delta=2\gamma -1$}

By using the Lie algebra $A^4_4$,
\begin{multline*}
	  \Span\left(\partial_{\tau},  e^{\frac{2 a}{(1-\gamma ) \rho ^2} \tau } \phi\partial _\phi, \frac{2 a}{(1-\gamma ) \rho ^2} \partial _x-B \phi\partial _\phi,\right.\\
	   \left. 2 e^{\frac{2 a}{(1-\gamma ) \rho ^2} \tau} \partial _x-e^{\frac{2 a}{(1-\gamma ) \rho ^2} \tau} \left(\frac{2 a}{(1-\gamma ) \rho ^2} x-2 B \tau\right)\phi \partial _\phi\right),
\end{multline*}
with $A=\frac{2 a}{(1-\gamma ) \rho ^2}$, we get the similarity solution
\begin{multline*}
	u(x,t)= x^{\frac{\gamma }{2}+\frac{\lambda }{\rho }} \exp \left(\frac{1}{8 \rho ^2}\left(\frac{B^2 (\gamma -1)^3 \rho ^8}{a^3}+\frac{4 B (\gamma -1) \rho ^4 x^{1-\gamma }}{a}-\right.\right.\\
	 4 B (\gamma -1)^2 \rho ^4  b^{1-\gamma } K^{1-\gamma } t e^{a (\gamma -1) (t-T)}+4 x^{1-2 \gamma } \left(\frac{2 \alpha }{2 \gamma -1}+\frac{\beta  x}{\gamma -1}\right)+\\
	\left.\left. \frac{4 a b^{1-2 \gamma } K^{1-2 \gamma } x^{-\gamma } e^{a (\gamma -1) (t-T)} \left(2 x b^{\gamma } K^{\gamma }-b K x^{\gamma } e^{a (\gamma -1) (t-T)}\right)}{\gamma -1}\right)\right),
\end{multline*}
for 
\begin{multline*}
	f(x,u) = \frac{x^{-2 (\gamma +1)u }}{8 (2 \gamma -1) \rho ^2} \left(4 a x^3 (2 \alpha -2 \gamma  (\alpha +\beta  x)+\beta  x)-\right.\\
	 4 a (\gamma -1) (2 \gamma -1) \rho  x^{2 \gamma +2} ((\gamma  \rho +2 \lambda )\log x -2 \rho  \log\lvert u\rvert)+\\
	(2 \gamma -1) \left(4 \rho  x^{2 \gamma +1} (2 \alpha  (\gamma  \rho +\lambda )+\beta  x ((2 \gamma -1) \rho +2 \lambda ))-4 B (\gamma -1)^2 \rho ^4 x^{\gamma +3}+\right.\\
	\left.\left.\rho ^2 x^{4 \gamma } \left((\gamma -2) \gamma  \rho ^2-4 \lambda ^2-4 \lambda  \rho \right)-4 x^2 (\alpha +\beta  x)^2\right)\right)
\end{multline*}	
 and $H(t) = b K e^{-a (t-T)}, a>0$.


\section{Conclusion}

In the present paper a generalization of a general bond-pricing equation,
\begin{equation*}
	u_t+\frac{1}{2}\rho^2x^{2\gamma} u_{xx}+(\alpha+\beta x-\lambda\rho x^\gamma)u_x-x u=0,
\end{equation*}
 was \changedText{suggested}\removedText{proposed} and studied under the view of the modern group analysis \changedText{with the intend  to use the obtained information to \emph{propose} nonlinear models with potential significance to Financial Mathematics}. To that end, we harnessed the advantage that the equivalence transformations offer when studying classes of differential equations: the knowledge  of the best representative(s) for this class of equations. This fact substantially simplifies the task of classifying it and obtaining its point symmetries.  

Through this classification interesting cases from the point of view of symmetries arise. Actually, the classical models of Financial Mathematics mentioned briefly in the introduction \removedText{are resurfaced} resurface, this further solidifies the fact that their significance and place in Financial Mathematics is well justified. Nonetheless the significance of the case $\delta = 2\gamma -1$,
$$
	u_t+\frac{1}{2}\rho^2x^{2\gamma} u_{xx}+(\alpha+\beta x-\lambda\rho x^{2\gamma -1})u_x-f(x,u)=0,
$$
is yet to be determined in the literature.

\changedText{Finally, with the help of equivalence transformations we prove that the whole class of equations of the form
\begin{equation*}
	u_t+\frac{1}{2}\rho^2x^{2\gamma} u_{xx}+(\alpha+\beta x-\lambda\rho x^\delta)u_x-f(x,u)=0,
\end{equation*}
can be transformed to the heat equation with nonlinear source,
\begin{equation*}
	\hat u_{\hat t}=\hat u_{\hat x\hat x}+\hat f(\hat x,\hat u).
\end{equation*}
A fact that greatly simplifies the complete group classification since now we have only one arbitrary element, the function $f(x,u)$.}

Nonlinear equations in general have few or no symmetries so cases that augment the set of symmetries at disposal are like an oasis in the desert. After all, it is evident in the related  literature that a dynamical system possessing an ample number of symmetries is more probable to relate with a physical system or model a more realistic process. \changedText{And indeed this is the case with the classical financial models: they keep appearing in group classifications of general models that include them, see for example  \cite{SiLeaHa2k8, MoKhaMo2k14, MaMaNaMo2k13}.}

This fact is even more decisive when we wish to study a boundary problem: since not all of the symmetries admit the boundary and its condition: some --- or \changedText{even} all --- of the symmetries \removedText{will}\changedText{could}  be excluded. Hence the bigger the set of symmetries the bigger the probability that some \removedText{will} survive the scrutiny of the boundary conditions and give an invariant solution for the problem in its entirety. 

This is evident for the terminal condition. As it can be seen by the examples  three- and four-dimensional algebras were used in order to arrive to a nontrivial solution. A lower dimensional algebra seems to be unable to yield a nontrivial solution. 

Things are different when \changedText{we consider} the barrier option \removedText{is considered}: because of the two arbitrary functions $H(t),R(t)$ a broader range of cases can yield interesting solutions. It is worth mentioning at this point that, as can be seen by the \changedText{solutions given in section $4$}, the barrier function $H$ usually used in the related literature is admitted by the symmetries. A fact that further strengthens the belief that  symmetries can be a valuable tool in investigating this kind of financial problems and the importance of this particular choice for the barrier function in Financial Mathematics.

Indeed, the insight provided through the above symmetry analysis might prove practical to anyone looking for a more realistic economic model without departing from the reasoning behind the proposed general bond-pricing model. \changedText{Indeed, not only we found specific cases that admit three- and four-dimensional Lie algebras but also constructed nontrivial solutions for the terminal boundary problem and the barrier option problem. Furthermore, the fact that those nonlinear equations are linked to the heat equation with a nonlinear source like the classical models to the heat equation reinforces that belief.}

\changedText{For all the above reasons it is our opinion that the }\removedText{The  provided}\changedText{suggested} nonlinear variant of \removedText{this model}\changedText{the general bond-pricing model --- along with its group classification ---} might be deemed useful in \changedText{Financial Mathematics}\removedText{that respect}.  \changedText{Moreover}\removedText{Even more}  when one studies more exotic kinds of options. Options that have gained ground in the Asian markets which in turn play an ever increasing role in the world market.   We leave to the interested reader the possible economical interpretation and use of the obtained results. 

\section*{Acknowledgements} 

S. Dimas is grateful to FAPESP (Proc. \#2011/05855-9)  for the financial support and IMECC-UNICAMP for their gracious hospitality. We would also like to thank Prof. Peter Leach for his helpful remarks and insight.

\bibliographystyle{model3-num-names}%
\bibliography{Bibliography.bib}

\begin{thebibliography}{22}
\providecommand{\natexlab}[1]{#1}
\providecommand{\url}[1]{\texttt{#1}}
\providecommand{\urlprefix}{URL }
\expandafter\ifx\csname urlstyle\endcsname\relax
  \providecommand{\doi}[1]{doi:\discretionary{}{}{}#1}\else
  \providecommand{\doi}{doi:\discretionary{}{}{}\begingroup
  \urlstyle{rm}\Url}\fi
\providecommand{\eprint}[2][]{\url{#2}}
\providecommand{\BIBand}{and}
\providecommand{\bibinfo}[2]{#2}
\ifx\xfnm\undefined \def\xfnm[#1]{\unskip,\space#1}\fi
\bibitem[{Sinkala et~al.(2008)Sinkala, Leach and O'Hara}]{SiLeaHa2k8}
\bibinfo{author}{Sinkala\xfnm[ W.]}, \bibinfo{author}{Leach\xfnm[ P.G.L.]},
  \bibinfo{author}{O'Hara\xfnm[ J.G.]}.
\newblock \bibinfo{title}{Invariance properties of a general bond-pricing
  equation}.
\newblock \bibinfo{journal}{J Diff Eq}
  \bibinfo{year}{2008};\bibinfo{volume}{244}(\bibinfo{number}{11}):\bibinfo{pages}{2820--2835}.
\bibitem[{Motsepa et~al.(2014)Motsepa, Khalique and Molati}]{MoKhaMo2k14}
\bibinfo{author}{Motsepa\xfnm[ T.]}, \bibinfo{author}{Khalique\xfnm[ C.M.]},
  \bibinfo{author}{Molati\xfnm[ M.]}.
\newblock \bibinfo{title}{Group {C}lassification of a {G}eneral
  {B}ond--{O}ption {P}ricing {E}quation of {M}athematical {F}inance}.
\newblock \bibinfo{journal}{Abstr Appl Anal}
  \bibinfo{year}{2014};\bibinfo{volume}{2014}.
\newblock \doi{\bibinfo{doi}{10.1155/2014/709871}}.
\bibitem[{Mahomed et~al.(2013)Mahomed, Mahomed, Naz and
  Momoniat}]{MaMaNaMo2k13}
\bibinfo{author}{Mahomed\xfnm[ F.M.]}, \bibinfo{author}{Mahomed\xfnm[ K.S.]},
  \bibinfo{author}{Naz\xfnm[ R.]}, \bibinfo{author}{Momoniat\xfnm[ E.]}.
\newblock \bibinfo{title}{Invariant approaches to equations of finance}.
\newblock \bibinfo{journal}{Math Comput Appl}
  \bibinfo{year}{2013};\bibinfo{volume}{18}(\bibinfo{number}{3}):\bibinfo{pages}{244--250}.
\bibitem[{Longstaff(1989)}]{Lo89}
\bibinfo{author}{Longstaff\xfnm[ F.A.]}.
\newblock \bibinfo{title}{A nonlinear general equilibrium model of the term
  structure of interest rates}.
\newblock \bibinfo{journal}{J Financ Econ}
  \bibinfo{year}{1989};\bibinfo{volume}{23}(\bibinfo{number}{2}):\bibinfo{pages}{195--224}.
\bibitem[{Vasicek(1977)}]{Va77}
\bibinfo{author}{Vasicek\xfnm[ O.]}.
\newblock \bibinfo{title}{An equilibrium characterization of the term
  structure}.
\newblock \bibinfo{journal}{J Financ Econ}
  \bibinfo{year}{1977};\bibinfo{volume}{5}(\bibinfo{number}{2}):\bibinfo{pages}{177--188}.
\bibitem[{Cox et~al.(1985)Cox, Ingersoll and Ross}]{CoIngeRo85}
\bibinfo{author}{Cox\xfnm[ J.C.]}, \bibinfo{author}{Ingersoll\xfnm[ J.E.]},
  \bibinfo{author}{Ross\xfnm[ S.A.]}.
\newblock \bibinfo{title}{An intertemporal general equilibrium model of asset
  prices}.
\newblock \bibinfo{journal}{Econometrica}
  \bibinfo{year}{1985};\bibinfo{volume}{53}(\bibinfo{number}{2}):\bibinfo{pages}{363--384}.
\bibitem[{Black and Scholes(1973)}]{BlaScho73}
\bibinfo{author}{Black\xfnm[ F.]}, \bibinfo{author}{Scholes\xfnm[ M.]}.
\newblock \bibinfo{title}{The pricing of options and corporate liabilities}.
\newblock \bibinfo{journal}{J Polit Econ}
  \bibinfo{year}{1973};\bibinfo{volume}{81}(\bibinfo{number}{3}):\bibinfo{pages}{637--654}.
\bibitem[{Olver(2000)}]{Olver2k}
\bibinfo{author}{Olver\xfnm[ P.J.]}.
\newblock \bibinfo{title}{Applications of Lie Groups to Differential
  Equations}; vol. \bibinfo{volume}{107} of \emph{\bibinfo{series}{Graduate
  Texts in Mathematics}}.
\newblock \bibinfo{address}{New York}: \bibinfo{publisher}{Springer};
  \bibinfo{edition}{2$^{nd}$} ed.; \bibinfo{year}{2000}.
\bibitem[{Popovych and Eshraghi(2005)}]{PoEshra2k5}
\bibinfo{author}{Popovych\xfnm[ R.O.]}, \bibinfo{author}{Eshraghi\xfnm[ H.]}.
\newblock \bibinfo{title}{Admissible point transformations of nonlinear
  {S}chr\"odinger equations}.
\newblock In: \bibinfo{editor}{Ibragimov\xfnm[ N.]},
  \bibinfo{editor}{Sophocleous\xfnm[ C.]}, \bibinfo{editor}{Damianou\xfnm[
  P.]}, editors. \bibinfo{booktitle}{Proceedings of the 10$^{th}$ International
  Conference in {MO}dern {GR}oup {AN}alysis}. \bibinfo{year}{2005}, p.
  \bibinfo{pages}{167--174}.
\bibitem[{Ivanova et~al.(2010)Ivanova, Popovych and Sophocleous}]{IvaPoSo2k10}
\bibinfo{author}{Ivanova\xfnm[ N.M.]}, \bibinfo{author}{Popovych\xfnm[ R.O.]},
  \bibinfo{author}{Sophocleous\xfnm[ C.]}.
\newblock \bibinfo{title}{Group analysis of variable coefficient
  diffusion-convection equations. {I}. {E}nhanced group classification}.
\newblock \bibinfo{journal}{Lobachevskii J Math}
  \bibinfo{year}{2010};\bibinfo{volume}{31}(\bibinfo{number}{2}):\bibinfo{pages}{100--122}.
\bibitem[{Black and Scholes(1972)}]{BlaScho72}
\bibinfo{author}{Black\xfnm[ F.]}, \bibinfo{author}{Scholes\xfnm[ M.]}.
\newblock \bibinfo{title}{The valuation of option contracts and a test of
  market efficiency}.
\newblock \bibinfo{journal}{J Finance}
  \bibinfo{year}{1972};\bibinfo{volume}{27}(\bibinfo{number}{2}):\bibinfo{pages}{399--417}.
\bibitem[{Merton(1974)}]{Me74}
\bibinfo{author}{Merton\xfnm[ R.C.]}.
\newblock \bibinfo{title}{On the pricing of corporate debt: The risk structure
  of interest rates}.
\newblock \bibinfo{journal}{J Finance}
  \bibinfo{year}{1974};\bibinfo{volume}{29}(\bibinfo{number}{2}):\bibinfo{pages}{449--470}.
\bibitem[{Ugur(2008)}]{Ugur2k8}
\bibinfo{author}{Ugur\xfnm[ O.]}.
\newblock \bibinfo{title}{Introduction to computational finance}.
\newblock \bibinfo{publisher}{Imperial College Press and World Scientific};
  \bibinfo{year}{2008}.
\bibitem[{O'Hara(2011)}]{ha2k11}
\bibinfo{author}{O'Hara\xfnm[ J.]}.
\newblock \bibinfo{title}{Lecture notes on exotic options}.
\newblock \bibinfo{year}{2011}.
\newblock \urlprefix\url{http://courses.essex.ac.uk/cf/cf966/}.
\bibitem[{Kwok(2008)}]{Kwo2k8}
\bibinfo{author}{Kwok\xfnm[ Y.K.]}.
\newblock \bibinfo{title}{Mathematical Models of Financial Derivatives}.
\newblock \bibinfo{publisher}{Springer}; \bibinfo{edition}{2nd} ed.;
  \bibinfo{year}{2008}.
\bibitem[{O'Hara et~al.(2013)O'Hara, Sophocleous and Leach}]{HaSoLea2k13}
\bibinfo{author}{O'Hara\xfnm[ J.G.]}, \bibinfo{author}{Sophocleous\xfnm[ C.]},
  \bibinfo{author}{Leach\xfnm[ P.G.L.]}.
\newblock \bibinfo{title}{Symmetry analysis of a model for the exercise of a
  barrier option}.
\newblock \bibinfo{journal}{Commun Nonlinear Sci Numer Simul}
  \bibinfo{year}{2013};\bibinfo{volume}{18}(\bibinfo{number}{9}):\bibinfo{pages}{2367--2373}.
\newblock \doi{\bibinfo{doi}{10.1016/j.cnsns.2012.12.027}}.
\bibitem[{Gazizov and Ibragimov(1998)}]{GazIbr98}
\bibinfo{author}{Gazizov\xfnm[ R.K.]}, \bibinfo{author}{Ibragimov\xfnm[ N.H.]}.
\newblock \bibinfo{title}{Lie symmetry analysis of differential equations in
  finance}.
\newblock \bibinfo{journal}{Non Dyn}
  \bibinfo{year}{1998};\bibinfo{volume}{17}(\bibinfo{number}{4}):\bibinfo{pages}{387--407}.
\bibitem[{Dimas et~al.(2009)Dimas, Andriopoulos, Tsoubelis and
  Leach}]{DiAndTsLe2k9}
\bibinfo{author}{Dimas\xfnm[ S.]}, \bibinfo{author}{Andriopoulos\xfnm[ K.]},
  \bibinfo{author}{Tsoubelis\xfnm[ D.]}, \bibinfo{author}{Leach\xfnm[ P.G.L.]}.
\newblock \bibinfo{title}{Complete specification of some partial differential
  equations that arise in financial mathematics}.
\newblock \bibinfo{journal}{J Nonlinear Math Phys}
  \bibinfo{year}{2009};\bibinfo{volume}{16, s-1}:\bibinfo{pages}{73--92}.
\bibitem[{Dimas(2008)}]{Dimas2k8}
\bibinfo{author}{Dimas\xfnm[ S.]}.
\newblock \bibinfo{title}{Partial differential equations, algebraic computing
  and nonlinear systems}.
\newblock \bibinfo{type}{Ph.{D}. {T}hesis}; University of Patras;
  \bibinfo{address}{Patras, Greece}; \bibinfo{year}{2008}.
\bibitem[{Bozhkov and Dimas(2014{\natexlab{a}})}]{BoDi2k14a}
\bibinfo{author}{Bozhkov\xfnm[ Y.]}, \bibinfo{author}{Dimas\xfnm[ S.]}.
\newblock \bibinfo{title}{Group classification of a generalized
  {B}lack--{S}choles--{M}erton equation}.
\newblock \bibinfo{journal}{Commun Nonlinear Sci Numer Simul}
  \bibinfo{year}{2014}{\natexlab{a}};\bibinfo{volume}{19}(\bibinfo{number}{7}):\bibinfo{pages}{2200--2211}.
\newblock \doi{\bibinfo{doi}{10.1016/j.cnsns.2013.12.016}}.
\bibitem[{Bozhkov and Dimas(2014{\natexlab{b}})}]{BoDi2k14b}
\bibinfo{author}{Bozhkov\xfnm[ Y.]}, \bibinfo{author}{Dimas\xfnm[ S.]}.
\newblock \bibinfo{title}{Group {C}lassification of a {G}eneralization of the
  {H}eath {E}quation}.
\newblock \bibinfo{journal}{Appl Math Comp}
  \bibinfo{year}{2014}{\natexlab{b}};\bibinfo{volume}{243}:\bibinfo{pages}{121--131}.
\newblock \doi{\bibinfo{doi}{10.1016/j.amc.2014.05.100}}.
\bibitem[{Zhdanov and Lahno(1999)}]{ZhdaLa99}
\bibinfo{author}{Zhdanov\xfnm[ R.Z.]}, \bibinfo{author}{Lahno\xfnm[ V.I.]}.
\newblock \bibinfo{title}{Group classification of heat conductivity equations
  with a nonlinear source.}
\newblock \bibinfo{journal}{J Phys A: Math Gen}
  \bibinfo{year}{1999};\bibinfo{volume}{32}(\bibinfo{number}{42}):\bibinfo{pages}{7405--7418}.

\end{thebibliography}

\end{document}